\documentclass[11pt,a4paper,reqno]{amsart}
\usepackage{amsmath}
\usepackage{amssymb,latexsym}
\usepackage[dvips]{graphicx}
\usepackage{color}
\newcommand{\R}{\mathbb R}

\newcommand{\N}{\mathbb N}

\newtheorem{theorem}{Theorem} [section]
\newtheorem{lemma}{Lemma} [section]

\newtheorem{corollary}{Corollary} [section]

\newtheorem{remark}{Remark}[section]

\begin{document}
\title [Generalization of Levitin and Parnovski inequality]{A generalization of a Levitin and Parnovski universal inequality for eigenvalues}
\author{ Sa\"{\i}d Ilias and Ola Makhoul}
\date{june 2009}
\address{S. Ilias, O. Makhoul: Universit\'e Fran\c{c}ois rabelais de Tours, Laboratoire de Math\'ematiques
et Physique Th\'eorique, UMR-CNRS 6083, Parc de Grandmont, 37200
Tours, France} \email{ilias@univ-tours.fr, ola.makhoul@lmpt.univ-tours.fr}


\begin{abstract} In this paper, we derive "universal" inequalities for the sums of eigenvalues of the Hodge de Rham Laplacian on Euclidean closed Submanifolds and of eigenvalues of the Kohn Laplacian on the Heisenberg group.  
These inequalities generalize the Levitin-Parnovski inequality obtained for the sums of eigenvalues of the Dirichlet Laplacian of a bounded Euclidean domain.
\end{abstract}

\maketitle

\section{Introduction}
Among inverse spectral problems, let us mention the following question (see for instance \cite{ColVerd})\\

\textit{"What kind of increasing sequences of non negative numbers can be the spectrum of the Laplacian of a compact Riemannian manifold (respectively of the Dirichlet Laplacian on a domain of a fixed Euclidean space) ?"}\\

This question can be asked in other more general contexts (for other operators and for Dirichlet or Neumann boundary conditions if the manifold has boundary, and for domains of a general Riemannian manifold instead of an Euclidean space).
Such sequences, which we will call spectral, admit some restrictions given by the asymptotics of Weyl and those of Minakshisundaram-Pleijel. Thereby, concerning those sequences, a natural question, less difficult than the first one, arises \\

\textit{"Is there any restrictions on these spectral sequences, which are independent of the manifold (respectively the domain) ? "}\\

Such restrictions will be called "universal".
The first result in this direction is the universal inequality of Payne, Polya and Weinberger \cite{PPW} obtained in 1955. In fact, they proved that the eigenvalues  $\{\lambda_i\}_{i=1}^{\infty}$ of the Dirichlet boundary problem for the Laplacian on a bounded domain $\Omega \subset \R^n$, must satisfy for each $k$,
\begin{equation} \label{ppw}
\displaystyle{ \lambda_{k+1}-\lambda_{k} \le \frac{4}{nk} \sum_{i=1}^{k} \lambda_{i}}
\end{equation}
(which we call henceforth the PPW inequality).\\
This result was improved in 1980 by Hile and Protter \cite{HileProt} (henceforth HP) who showed that, for $k=1,2,\ldots$
\begin{equation}
\displaystyle \frac{nk}{4} \leq \sum_{i=1}^k
\frac{\lambda_i}{\lambda_{k+1}-\lambda_i}.
\end{equation}
In 1991, H.C.Yang (see \cite{Yang.HC} and more recently
\cite{ChengYang1}) proved
\begin{equation}\label{1}
\displaystyle \sum_{i=1}^k (\lambda_{k+1}-\lambda_i)^2 \leq
\frac{4}{n} \sum_{i=1}^k \lambda_i(\lambda_{k+1}-\lambda_i),
\end{equation}
which is, until now, the best improvement of the PPW inequality (see for instance \cite{Ashb1} for a comparison of all these three inequalities).\\

Apart from this class of inequalities (PPW, HP and Yang) which was intensively studied, there exists another class, much less known, discovered by Levitin and Parnovski (see Example 4.2 and identity (4.14) of \cite{LevPar}). Indeed, they proved for the eigenvalues of the Dirichlet Laplacian of any bounded domain of $\R^{n}$ and for any $k$,
\begin{equation}\label{Levi-Parn}
\sum_{i=1}^{n} \lambda_{k+i} \le (4+n) \lambda_{k}
\end{equation}
(these inequalities, indexed by $k$, will be referred to henceforth as Levitin and Parnovski inequalities).\\
These inequalities generalize a previous inequality obtained for $k=1$ by PPW
\cite{PPW} in dimension $n=2$ and by Ashbaugh (cf section 3.2 of \cite{Ashb2}) in all dimensions.\\
\indent All these universal inequalities show that one can not prescribe arbitrarily a finite part of the spectrum of the Dirichlet Laplacian on a bounded domain of an Euclidean space. This contrasts completely with the situation of the Laplace operator on a compact manifold (or the Neumann Laplacian on a bounded Euclidean domain), for which Colin de Verdi\`ere \cite{ColVerd} showed that it is possible to prescribe any finite part of the spectrum. More precisely, Colin de Verdi\`ere proved that, if $s_{N}=\{ \lambda_{1}=0 <\lambda_{2}\le \dots \le \lambda_{N}\}$ is a finite set of real numbers and if $M$ is a compact manifold without boundary of dimension $\ge 3$, then there exists a Riemannian metric on $M$ having $s_{N}$ as the beginning of the spectrum of its Laplacian. This was generalized by Guerini \cite {Guerini} to the Hodge de Rham Laplacian acting on differential forms for compact manifolds without boundary and for bounded Euclidean domains with the relative or absolute boundary conditions. As a consequence of these prescription results of a part of the spectrum, contrary to the situation of the Dirichlet Laplacian acting on functions on Euclidean bounded domains, one cannot expect a universal inequality for the Laplacian  and more generally for the Hodge de Rham Laplacian acting on forms, on a compact Riemannian manifold.
However, a generalization of the PPW universal inequality holds for some special manifolds. In fact, in 1975, Cheng \cite{Cheng} showed that the PPW inequality (\ref{ppw}) holds for domains of minimal hypersurfaces of $\R^{n+1}$ (note that his proof works also for codimension $\ge 1$). In the same spirit, Yang and Yau \cite{YangYau} obtained a generalization of the PPW inequality for the eigenvalues of the Laplacian of any compact minimal Submanifold of a Sphere. Note that these two results indicate that, a role must probably be played by the extrinsic geometry of the Submanifolds in an eventual generalization of the PPW inequality.  Other generalizations were obtained (see for instance \cite{Anghel}, \cite{Ashb1}, \cite{AshbHer1}, \cite{AshbHer3}, \cite{Cheng}, \cite{Col}, \cite{Harl1}, \cite{Harl2}, \cite{HarlMichel2}, \cite{HarlMichel1}, \cite{harlStub2}, \cite{HarlStub}, \cite{HileProt},\cite{IlMa}, \cite{Lee}, \cite{LeungPF2}, \cite{LiP}, \cite{Soufi.Harl.Ilias} and \cite{Yang.HC}), among them we mention the results of Lee \cite{Lee} and Anghel \cite{Anghel} which constitute a first tentative to a generalization of the PPW inequality to the eigenvalues of the Hodge de Rham Laplacian on an Euclidean compact Submanifold. Unfortunately, these generalized inequalities depend on the intrinsic geometry of the Submanifold. Nevertheless, the results of Colin de Verdi\`ere, Guerini, Cheng and Yang and Yau, suggest in the case of Euclidean Submanifolds the following question \\

\textit{"Can one find Universal inequalities of PPW, HP, Yang or Levitin and Parnovski type for the eigenvalues of the Dirichlet Laplacian on a bounded domain of an Euclidean Submanifold or for the eigenvalues of the Hodge de Rham Laplacian on an Euclidean compact Submanifold,  which depends only on the extrinsic geometry of the Submanifold (i.e its second fundamental form or its mean curvature) ? "}\\

\indent Using an algebraic commutation inequality of Harrell and Stubbe \cite{HarlStub}, we gave in \cite{IlMa}(see also the references therein for partial results) a complete answer to the first part of the question, concerning PPW, HP and Yang type inequalities. In the present article we will focus on the second part of the question. Using an algebraic identity obtained by Levitin and Parnovski and by a method completely different to that we used in \cite{IlMa}, we will give a positive answer to the second part of the question which extends the Levitin and Parnovski inequalities (\ref{Levi-Parn}) to the eigenvalues of the Hodge de Rham Laplacian of a compact Euclidean Submanifold . We observe that our proof works also for the eigenvalues of the Dirichlet Laplacian on bounded domains of Euclidean Submanifolds.\\
We must note that some partial generalizations of the Levitin and Parnovski inequality was obtained recently by Chen and Cheng \cite{ChengChen} and by Sun, Cheng and Yang \cite{ChengYang3}. But, it turns out that all these generalizations are particular cases of our results. Indeed, on one hand, a direct consequence of our work (apply Corollary \ref{cor theorem 1} with $q=0$) is that, for any bounded domain $\Omega$ of an $m$-dimensional isometrically immersed Riemannian manifold $M$ in an Euclidean space and for any $k$, we have 
 \begin{equation}\label{LP1}
\sum_{i=1}^{m} \lambda_{k+i} \le (4+m) \lambda_{k}+\left\|H\right\|^{2}_{\infty,\Omega}
\end{equation}   
where $\left\{\lambda_{j}\right\}_{j=1}^{\infty}$ are the eigenvalues of the Dirichlet Laplacian of $\Omega$, $H$ is the mean curvature vector of the immersion of $M$ (i.e the trace of its second fundamental form) and $\left\|H\right\|^{2}_{\infty,\Omega}=\displaystyle{\sup_{\Omega}|H|^{2}}$.
When we take $k=1$ in this inequality (\ref{LP1}), we obtain as a direct consequence the generalization obtained by Chen and Cheng (see Theorem 1.1 of \cite{ChengChen}). On the other hand, if we combine inequality (\ref{LP1}) with the standard embeddings of the compact rank one symmetric spaces in an Euclidean space, we easily derive a similar inequalities for Submanifolds of a Sphere or a Projective space. Let us denote by $\overline{M}$ the Sphere $\mathbb{S}^{n}$, the real projective space $\mathbb{R}P^{n}$, the complex projective space $\mathbb{C}P^{n}$ or the quaternionic projective space $\mathbb{Q}P^{n}$ endowed with their respective standard metrics and let $M$ be an $m$-dimensional Riemannian manifold isometrically immersed in $\overline{M}$. We prove (see Corollary \ref{corRSS}) that for any bounded bounded domain of $M$ and for any $k \ge 1$,
 \begin{equation}\label{LP2}
\displaystyle \sum_{i=1}^m \lambda_{k+i} \le (4+m) \lambda_{k}+\big(\left\| H \right\|^{2}_{\infty,\Omega}+d(m)\big)
\end{equation}
 where, $H$ is the mean curvature of $M$ in $\overline{M}$ and $d(m)$ is the constant given by 
\begin{equation*}
d(m)=
    \begin{cases}
    m^{2},     &\text{if $\overline{M}=\mathbb{S}^{n}$}\\
    2m(m+1),   &\text{if $\overline{M}= \mathbb{R}P^{n}$}\\
    2m(m+2),   &\text{if $\overline{M}= \mathbb{C}P^{n}$}\\
    2m(m+4),    &\text{if $\overline{M}= \mathbb{Q}P^{n}$}.\\

    \end{cases}
\end{equation*}
If we apply inequality (\ref{LP2}) with $k=1$ to domains of $\mathbb{S}^{n}$ or $\mathbb{C}P^{n}$ (respectively for complex Submanifolds of $\mathbb{C}P^{n}$ which are in particular minimal), then we obtain Theorem 1.1 and Theorem 1.3 (respectively Theorem 1.2) of Sun, Cheng and Yang \cite{ChengYang3}. 
\indent Another consequence of our work is an extension to the eigenvalues of higher order of the Reilly inequality (respectively the Asada inequality) concerning the first eigenvalue of the Laplacian (respectively the Hodge de Rham Laplacian) of a compact Riemannian manifold isometrically immersed in an Euclidean space. Indeed, for any $m$-dimensional compact Riemannian manifold immersed in an Euclidean space, Reilly \cite{Reilly} proved the following inequality between the first positive eigenvalue of its Laplacian and the mean curvature of its immersion
 \begin{equation}\label{rel}
 \displaystyle \lambda_{2}\leq \frac{1}{m\,{\rm Vol}(M)}\int_{M} |H|^2 dV_M
\end{equation}
where $dV_M$ and ${\rm Vol}(M)$ are respectively the Riemannian volume element and the volume of $M$.
Asada \cite{Asada} obtains an extension of this inequality to the first positive eigenvalue of the Hodge de Rham Laplacian acting on $p$-forms
  \begin{equation}\label{as}
 \lambda_1^{(p)}(M) \leq \frac{p}{m(m-1) {\rm Vol}(M)} \int _M
\Big[(m-p)|H|^2+(p-1)|h|^2\Big] dV_M.
\end{equation}
where $h$ denotes the second fundamental form of the immersion of $M$. To be more precise, we note that, Asada proves more. In fact, he proves this inequality for the first positive eigenvalue of the Hodge de Rham Laplacian restricted to the closed $p$-forms.\\
Using our generalizations of the Levitin and Parnovski universal inequalities, one can easily extends the Reilly and the Asada inequalities to all the eigenvalues of the Laplacian and the Hodge de Rham Laplacian of Euclidean Submanifolds. We derive (see Corollary \ref{cor Reilly general}) in particular the surprising generalization of the Reilly inequality
\begin{equation*}
 \displaystyle \sum_{k=1}^m \lambda_{k+1}\leq \frac{1}{Vol(M)}\int_{M} |H|^2 dV_M.
\end{equation*}
\indent We limit ourselves to the case of the Hodge de Rham Laplacian, but all our arguments work with minor modifications in the setting of general Laplace operators on Riemannian fiber bundles. \\
\indent Another different situation which is not Riemannian involving an operator which is not elliptic, is that of the Kohn Laplacian on the Heisenberg group. In the second section, we derive a Levitin and Parnovski inequality in this case. 

\section{Generalization of the Levitin-Parnovski inequality to the Hodge de Rham Laplacian}

Let $(M,g)$ be an $m$-dimensional compact Riemannian manifold. We denote by $\bigwedge^{p}(M)$ for $p \in \left\{0, \dots ,m \right\}$, and by $\Gamma(TM)$ respectively, the space of smooth differential $p$-forms and the space of smooth vector-fields of $M$.\\
For any two $p$-forms $\alpha$ and $\beta$, we let $\alpha_{i_1i_2,\ldots,i_p}=\alpha(e_{i_1},e_{i_2},\ldots,e_{i_p})$ and $\beta_{i_1i_2,\ldots,i_p}=\beta(e_{i_1},e_{i_2},\ldots,e_{i_p})$ denote the components of $\alpha$ and $\beta$, with respect to a local orthonormal frame $(e_i)_{i \leq m}$. Their pointwise inner product with respect to $g$ is given by
\begin{equation*}
\langle \alpha,\beta \rangle = \frac{1}{p!} \sum_{1 \le i_{1},\dots,i_{p}\le m}\alpha_{{i_1},\dots,{i_p}}\; \beta_{{i_1},\dots,{i_p}}.
\end{equation*}
We denote by $\Delta_{p}$ the Hodge de Rham Laplacian acting on $p$-forms  
\begin{equation*}
\Delta_{p} :=(d \, \delta + \delta  d),
\end{equation*}
where $d$ is the exterior derivative acting on $p$-forms and $\delta$ is the adjoint of $d$ with respect to the $L_{2}(g)$ global inner product.\\
The spectrum of  $\Delta_{p}$ consists of a nondecreasing, unbounded sequence of eigenvalues with finite multiplicities
$${\rm Spec}(\Delta_{p})=\{0 \le \lambda_{1}^{(p)} \le \lambda_{2}^{(p)} \le \lambda_{3}^{(p)} \le \cdots \le \lambda_{i}^{(p)} \le \cdots \}.$$
If we denote by $\nabla$ the extension to $p-$forms of the Levi-Civita connexion of $(M,g)$ and by $\nabla^{\ast}$ its formal adjoint with respect to the metric $g$, then the Bochner-Weitzenb\"{o}ck formula gives for any $\alpha  \in\bigwedge^{p}(M)$
\vspace{0.3cm}
\begin{center}
 $\Delta_{p} \, \alpha = \nabla^{\ast} \nabla \alpha +\mathcal{R}_{p}(\alpha)$
\end{center}
where $\mathcal{R}_{p}$ is the curvature term which is a selfadjoint endomorphism of $\bigwedge^{p}(M)$ defined  for any  $X_{1},\dots,X_{p} \in \Gamma(TM)$ by
\begin{align*}
\mathcal{R}_{p}(\alpha)(X_{1},\dots,X_{p}) &
=\sum_{i,j}(-1)^{i} i_{e_{j}}(R(e_{j},X_{i})\alpha)(X_{1},\dots,\hat{X_{i}},\dots,X_{p}),\\
\end{align*}
here $(e_{i})_{i \le m}$ is a local orthonormal frame as before and $R$ is the extension of the curvature tensor to forms which is given for $X,\,Y \in \Gamma(TM)$, by
\begin{equation*}  R(X,Y)\alpha= \nabla_{\left[X,Y\right]}\alpha - \left[\nabla_{X},\nabla_{Y}\right]\alpha, 
\end{equation*}
An immediate consequence of the Bochner-Weitzenb\"{o}ck formula is
the following
\begin{equation}\label{weitzenbock}
\langle \Delta_{p}\alpha,\alpha \rangle= |\nabla \alpha|^{2}+\frac{1}{2} \Delta |\alpha|^{2}+\langle \mathcal{R}_{p}(\alpha),\alpha \rangle.
\end{equation}
\\
\indent In this section, the main objective is to extend the universal inequality of Levitin and Parnovski (see inequality (4.14) in \cite{LevPar}) concerning the eigenvalues of the
Dirichlet Laplacian on bounded Euclidean domains to the eigenvalues of
the Hodge-de Rham Laplacian on closed Euclidean Submanifolds.\\

\begin{theorem} \label{theorem 1} Let $X:(M^m,g)\longrightarrow (\R^n,{\rm can})$ be an isometric
immersion and $H$ be its mean curvature vector field (i.e. the trace of its second fundamental form $h$). We have, for any
$p \in \left\{1,\dots,m\right\}$ and $j \in \N^{\ast}$,
\begin{equation}\label{ineq1}
\displaystyle{\sum_{l=1}^m \lambda_{j+l}^{(p)}}\leq
\displaystyle{4\bigg[\Big(1+\frac{m}{4}\Big)\lambda_j^{(p)}
-\int_{M}\langle
\mathcal{R}_p(\omega_j),\omega_j\rangle+\frac{1}{4}\int_{M}|H|^2|\omega_j|^2\bigg]},
\end{equation}
where $\Big\{\lambda_j^{(p)}\Big\}_{j=1}^\infty$ are the eigenvalues of $\Delta_p$ and $\{\omega_j\}_{j=1}^\infty$ is a corresponding orthonormal basis of $p-$eigenforms.
\end{theorem}

\begin{proof}[Proof of Theorem \ref{theorem 1}] To prove this inequality, we need the
following algebraic identity obtained by Levitin and Parnovski (see identity 2.2 of Theorem 2.2 in \cite {LevPar}).

\begin{lemma}\label{levitin}
Let $L$ and $G$ be two self-adjoint operators with domains $D_{L}$ and $D_{G}$ contained in a same Hilbert space and such that
$G(D_L)\subseteq D_L \subseteq D_G$. Let $\lambda_j$ and $u_j$
be the eigenvalues and orthonormal eigenvectors of $L$. Then, for each $j$,
\begin{equation*} \displaystyle{\sum_{k}
\frac{|\langle[L,G]u_j,u_k\rangle|\;^2}{\lambda_k-\lambda_j}}=\displaystyle{-\frac{1}{2}\langle[[L,G],G]u_j,u_j\rangle}
\end{equation*}
(The summation is over all $k$ and is correctly defined even when
$\lambda_{k}=\lambda_{j}$ because in this case $\langle[L,G]u_j,u_k\rangle=0$
(see Lemma 2.1 in \cite{LevPar})).

\end{lemma}

Now (\ref{ineq1}) will follow by applying this Lemma \ref{levitin} with
$L=\Delta_p$ and $G=X_l$, where $X_l$ is one of the components
$(X_1,...,X_n)$ of the isometric immersion $X$. First, we have
\begin{align*}
{}\displaystyle{[[\Delta_p,X_l],X_l]\omega_j} & =  [\Delta_p,X_l](X_{l}\omega_{j})-X_{l}([\Delta_p,X_{l}]\omega_{j}) \\
{} & \displaystyle = (\Delta X_{l})(X_{l}\omega_{j})-2\nabla_{\nabla X_{l}}(X_{l}\omega_{j})\\
{} & \quad \displaystyle -X_{l}\Big((\Delta X_{l})\omega_{j}-2\nabla_{\nabla X_{l}}\omega_{j}\Big)\\
{} & \displaystyle = -2 |\nabla X_{l}|^{2} \omega_{j},
\end{align*}
hence\begin{equation*}
\displaystyle{-\frac{1}{2}\langle[[\Delta_p,X_l],X_l]\omega_j,\omega_j\rangle_{L^{2}}=\int_{M}|\nabla
X_l|^2|\omega_j|^2}. \end{equation*} Thus, Lemma \ref{levitin} gives

\begin{equation}\label{i}
\displaystyle{\sum_{k }\frac{\Big(\displaystyle\int_{M}\langle[\Delta_p,X_l]\omega_j,\omega_k\rangle\Big)^2}{\lambda_k^{(p)}-\lambda_j^{(p)}}=\int_{M}|\nabla
          X_l|^2|\omega_j|^2}.
\end{equation}
Now for a fixed $j$, let $A$ be the matrix
\begin{equation*}
\Big(\omega_{k,\;l}=\int_{M}\langle[\Delta_p, X_l]\omega_j,\omega_{j+k}\rangle\Big)_{1\leq k,\;l \leq n}.
\end{equation*}

Applying Gram-Schmidt orthogonalization, we can find an orthogonal
coordinate system such that $A$ has the following triangular form,

$$\left(\begin{array}{ccccc}
                                    \omega_{1,1} \\
                                     \omega_{2,1} &\omega_{2,2}&\text{{\huge{0}}}\\
                                    \vdots &  & \ddots \\
                                    \omega_{n,1}&\cdots &\cdots& \omega_{n,n}
\end{array}\right)$$
where $\omega_{k,\;l}=0\;\;if\;\;k<l$.\\
Equation (\ref{i}) can be written as follows
\begin{align}\label{j}
{} \int_{M}|\nabla X_l|^2|\omega_j|^2 & =  \sum_{k=1}^{j-1}\frac{\Big(\displaystyle \int_{M}\langle[\Delta_p,X_l]\omega_j,\omega_k\rangle\Big)^2}{\lambda_k^{(p)}-\lambda_j^{(p)}} \nonumber\\
{} & +\sum_{k=j+1}^{j+l-1}\frac{\Big(\displaystyle\int_{M}\langle[\Delta_p,X_l]\omega_j,\omega_k\rangle\Big)^2}{\lambda_k^{(p)}-\lambda_j^{(p)}} \nonumber\\
{} &
+\sum_{k=j+l}^{\infty}\frac{\Big(\displaystyle\int_{M}\langle[\Delta_p,X_l]\omega_j,\omega_k\rangle\Big)^2}{\lambda_k^{(p)}-\lambda_j^{(p)}}.
\end{align}
The first term of the right-hand side of equality (\ref{j}) is
nonpositive because $k <j$. The second term is equal to
\begin{equation*}\displaystyle{\sum_{k=1}^{l-1}\frac{\Big(\displaystyle\int_{M}\langle[\Delta_p,X_l]\omega_j,\omega_{j+k}\rangle\Big)^2}{\lambda_{j+k}^{(p)}-\lambda_j^{(p)}}}
\end{equation*}
which is equal to zero because $\displaystyle
\int_{M}\langle[\Delta_p,X_l]\omega_j,\omega_{j+k}\rangle=\omega_{k,\;l}=0$
if $k<l$, \\
therefore
\begin{align}\label{k}
{} \displaystyle{\int_{M}|\nabla X_l|^2|\omega_j|^2} & \leq \displaystyle{\sum_{k=j+l}^{\infty}\frac{\Big(\displaystyle\int_{M}\langle[\Delta_p,X_l]\omega_j,\omega_k\rangle\Big)^2}{\lambda_k^{(p)}-\lambda_j^{(p)}}} \nonumber\\
{} & \leq
\displaystyle{\frac{1}{\lambda_{j+l}^{(p)}-\lambda_j^{(p)}}\sum_{k=1}^\infty\Big(\int_{M}\langle[\Delta_p,X_l]\omega_j,\omega_k\rangle\Big)^2}.
\end{align}
Parceval's identity implies that
\begin{equation} \label{Parc}
\displaystyle \sum_{k=1}^\infty\Big(\int_{M}\langle[\Delta_p,X_l]\omega_j,\omega_k\rangle\Big)^2=\|[\Delta_p,X_l]\omega_j\|^2_{L^2}.
 \end{equation}
Hence using (\ref{k}), (\ref{Parc}) and summing on $l$, we obtain
\begin{align}\label{l}
{} \sum_{l=1}^{n}
\Big(\lambda_{j+l}^{(p)}-\lambda_j^{(p)}\Big)\bigg(\int_{M}|\nabla
X_l|^2|\omega_j|^2\bigg) & \leq \sum_{l=1}^n \|[\Delta_p,X_l]\omega_j\|^2_{L^2}.
\end{align}
We need now to calculate $\sum_{l=1}^n \|[\Delta_p,X_l]\omega_j\|^2_{L^2}$. First, we have
\begin{equation*}
 \left[\Delta_{p},X_l \right]=\left[\nabla^{\ast}\nabla,X_l \right],
\end{equation*}
because the curvature term in the Bochner-Weitzenb\"{o}ck formula is $\mathcal{C}^{\infty}(M)$-linear. Then, at a point $x \in M$, we take a local orthonormal frame $(e_{i})_{i \le m}$ of $M$ which
is normal at $x$. We have at $x$
\begin{align*}
 \left[\Delta_{p}, X_l\right]\omega_j & = \nabla^{\ast}\nabla
(X_l \omega_j)-X_l \nabla^{\ast}\nabla \omega_j \\
& = - \sum_{i \le m} \nabla_{e_{i}}\nabla_{e_{i}}(X_l \omega_j)-X_l \nabla^{\ast}\nabla \omega_j\\
& = -\sum_{i\le
m}e_{i}(e_{i}(X_l))\omega_j-2\nabla_{\nabla
X_l}\omega_j+X_l\nabla^{\ast}\nabla \omega_j-X_l \nabla^{\ast}\nabla \omega_j\\
& =(\Delta X_l)\omega_j-2\nabla_{\nabla X_l}\omega_j.
\end{align*}
Hence we obtain
\begin{equation*}
\left\|\left[\Delta_p,X_l\right]\omega_j\right\|^{2}_{L^{2}}= \int_{M}
(\Delta X_l)^{2}|\omega_j|^{2} + 4 \int_{M} |\nabla_{\nabla X_l}\omega_j|^{2}-4 \int_{M}
\langle (\Delta X_l) \omega_j,\nabla_{\nabla X_l}\omega_j \rangle.
\end{equation*}
Since $X$ is an isometric immersion, we have $\sum_{l \le n} |\nabla_{\nabla
X_{l}}\omega_{j}|^{2}=|\nabla \omega_{j}|^{2}$, ${(\Delta X_{1},\dots,\Delta
X_{n})=H}$ and $\sum_{l \le n}\langle (\Delta X_{l}) \omega_{j},\nabla_{\nabla X_{l}}\omega_{j}\rangle=\frac{1}{2} \langle H,\nabla |\omega_{j}|^{2}\rangle=0.$\\
Thus it follows that 
\begin{align}\label{c}
{} \sum_{l \le n}\left\|\left[\Delta_p,X_{l}\right]\omega_{j}\right\|^{2}_{L^{2}}=&\int_{M} \sum_{l \le n}(\Delta X_{l})^{2}|\omega_{j}|^{2} + 4 \int_{M} \sum_{l \le n}|\nabla_{\nabla X_{l}}\omega_{j}|^{2} \nonumber\\
{} & \quad -4\int_{M} \sum_{l \le n}\langle(\Delta X_{l})\omega_{j},\nabla_{\nabla X_{l}}\omega_{j}\rangle \nonumber \\
{}=& \int_{M} |H|^{2} |\omega_{j}|^{2}+4 \int_{M} |\nabla \omega_{j}|^{2}\nonumber\\
{}=& \int_{M} |H|^{2} |\omega_{j}|^{2}+4\lambda_{j}^{(p)}-4\int_{M} \langle \mathcal{R}_p(\omega_{j}),\omega_{j}\rangle. 
\end{align}
The last equality follows from the Bochner-Weitzenb\"{o}ck formula, in fact we have
\begin{equation*}
\int_{M} |\nabla \omega_{j}|^{2}= \int_{M} \langle \Delta_p
\omega_{j},\omega_{j} \rangle-\langle
\mathcal{R}_p(\omega_{j}),\omega_{j}\rangle =\lambda_{j}^{(p)}-\int
_{M}\langle \mathcal{R}_p(\omega_{j}),\omega_{j}\rangle.
\end{equation*}
On the other hand, since the immersion $X$ is isometric we have $\displaystyle{\sum_{l=1}^{n}|\nabla X_l|^2=m}$ and therefore
\begin{align}\label{m}
{}\sum_{l=1}^n \Big(\lambda_{j+l}^{(p)}-\lambda_j^{(p)}\Big)\bigg(\int_{M}|\nabla
X_l|^2|\omega_j|^2\bigg)&=\sum_{l=1}^n\bigg(\int_{M}|\nabla X_l|^2|\omega_j|^2\bigg)\lambda_{j+l}^{(p)} \nonumber\\
{} & -m\lambda_j^{(p)},
\end{align}
then we obtain, from (\ref{l}), (\ref{c}) and (\ref{m})
\begin{align}\label{n}
{} \sum_{l=1}^{n}\bigg(\int_{M}|\nabla X_l|^2|\omega_j|^2\bigg)\lambda_{j+l}^{(p)} & \leq  (4+m)\lambda_j^{(p)}-4\int_{M}\langle \mathcal{R}_p(\omega_j),\omega_j\rangle \nonumber\\
{}  & +\int_{M}|H|^2|\omega_j|^2.
\end{align}
To finish the proof, we will show that
\begin{equation}\label{o}
\sum_{l=1}^n |\nabla X_l|^2 \lambda_{j+l}^{(p)}  \ge \sum_{l=1}^m  \lambda_{j+l}^{(p)}.
\end{equation}
In fact, let us prove inequality (\ref{o}) at an arbitrary $x\in M$. Denote by $(\epsilon_{i})_{i\le n}$ the standard Euclidean basis of $\R^{n}$. Since the immersion $X$ is isometric, we deduce that: there exist $\,l_{1},\dots, l_{m} \in \left\{1,\dots,n\right\}$ and $ i_{1},\dots, i_{m} \in \left\{1,\dots,n\right\}$ such that
\begin{itemize}
\item For any $ k \in \left\{1,\dots,m\right\}:\quad l_{k} \ge k.$
\item For any $k \in \left\{1,\dots,m\right\}: \quad \nabla X_{l_{k}}(x)=\epsilon_{i_{k}}.$
\item For any $i \notin \left\{l_{1},\dots,l_{m}\right\}:\quad \nabla X_{i}(x)=0.$
\end{itemize}
Therefore, we have at $x$
\begin{equation*}
 \sum_{l=1}^n |\nabla X_l|^2 \lambda_{j+l}^{(p)}= \sum_{k=1}^{m}\lambda_{j+l_{k}}^{(p)} \ge \sum_{l=1}^m  \lambda_{j+l}^{(p)}
\end{equation*}
which proves inequality (\ref{o}).
\\
Finally, we deduce, from (\ref{n}) and (\ref{o}), that
\begin{equation*}
\displaystyle{\sum_{l=1}^m \lambda_{j+l}^{(p)}\leq
(4+m)\lambda_j^{(p)}-4\int_{M}\langle
\mathcal{R}_p(\omega_j),\omega_j\rangle+\int_{M}|H|^2|\omega_j|^2}.
\end{equation*}

\end{proof}
\begin{remark}
\begin{enumerate}
\item Note that our result does not depend on the dimension of the ambient space $\R^{n}$.
\item We observe that, the same ideas work for general operators of Laplace type acting on the sections of a Riemannian vector bundle on $M$ endowed with a Riemannian connexion.
\end{enumerate}
\end{remark}
\begin{corollary} \label{cor1}
Under the conditions of Theorem \ref{theorem 1} , we have for any $j \ge 1$,
\begin{equation*}
\sum_{l=1}^m \lambda_{j+l}^{(p)}\leq 4\bigg((1+\frac{m}{4})\lambda_j^{(p)}-\delta_1+\frac{1}{4}\delta_2\bigg),
\end{equation*}
where $\displaystyle \delta_1=\inf_{x\in M}\tilde{\mathcal{R}}_{p}(x)$,  $\tilde{\mathcal{R}}_{p}(x)$ being the smallest eigenvalue, at $x \in M$, of the endomorphism $(\mathcal{R}_p)_x$ of
$\bigwedge^p (T_x M)$, and $\delta_2=\sup |H|^{2}$.
\end{corollary}

The inequalities of Theorem \ref{theorem 1} and Corollary \ref{cor1} depends on the intrinsic geometry of the Submanifold $M$, because it involves the curvature term $\mathcal{R}_p$. Using as in Theorem 3.2 of  \cite{IlMa}, the extrinsic estimate we derived for $\mathcal{R}_p$ in terms of the second fundamental form $h$ and the mean curvature $H$ of the immersion $X$ of $M$, we obtain 
\begin{theorem} \label{PPW} Under the conditions of Theorem \ref{theorem 1}, we have for any $j \ge 1$,
\begin{equation}\label{q}
{} \displaystyle \sum_{l=1}^m \lambda_{j+l}^{(p)} \leq
4\bigg\{\Big(1+\frac{m}{4}\Big)\lambda_j^{(p)}+\int_{M}\phi(h,H)|\omega_j|^2\bigg\},
\end{equation}
where \begin{align*}
\phi(h,H)=&p^2\bigg[\Big(\frac{m-5}{4}\Big)|H|^2+|h|^2-\frac{1}{4m^2}\Big(\sqrt{m-1}(m-2)|H|\\
{}&-2\sqrt{m|h|^2-|H|^2}\,\Big)^2
\bigg]+\frac{1}{2}\sqrt{p}(p-1)\Big(|H|^2+|h|^2\Big)+\frac{1}{4}|H|^2.
\end{align*}
\end{theorem}
\begin{proof}[Proof of Theorem \ref{PPW}]
This theorem follows immediately from the estimate of $\mathcal{R}_p$ obtained in Theorem 3.2 of \cite{IlMa} 
\begin{align*}
\langle\mathcal{R}_p(\omega_j),\omega_j\rangle \ge& \bigg\{-p^2\bigg[\Big(\frac{m-5}{4}\Big)|H|^2+|h|^2-\frac{1}{4m^2}\Big(\sqrt{m-1}(m-2)|H|\\
{}&-2\sqrt{m|h|^2-|H|^2}\,\Big)^2\bigg]-\frac{1}{2}\sqrt{p}(p-1)\Big(|H|^2+|h|^2\Big)\bigg\}|\omega_j|^2.
\end{align*}
\end{proof}

One can obviously eliminate the dependence on $\omega_j$ by taking
the supremum of $\phi(h,H)$, and obtain the following extension of the Levitin and Parnovski inequality which depends only on extrinsic invariants of the Submanifold $M$,
\begin{corollary} \label {corLP}
 Under the conditions of Theorem \ref{theorem 1}, $\forall j \ge 1$ we have
\begin{equation}\label{r}
\displaystyle \sum_{l=1}^m \lambda_{j+l}^{(p)} \leq 4 \bigg\{ \Big(1+\frac{m}{4}\Big)\lambda_j^{(p)}+ \|\phi(h,H)\|_{\infty} \bigg\},
\end{equation}
where  $\phi(h,H)$ is as in Theorem \ref{PPW}.
\end{corollary}

In the particular case where $j=1$, we obtain
\begin{corollary} Under the
conditions of Theorem \ref{theorem 1}, we have
\begin{equation*}
\displaystyle \sum_{l=1}^m \lambda_{l+1}^{(p)} \leq
\Phi(h,H),
\end{equation*}
where
\begin{align*}
{}\Phi(h,H)=4\bigg\{\frac{p}{m(m-1)Vol(M)}& \Big(1+\frac{m}{4}\Big)\int_M
\Big[(m-p) |H|^2+(p-1)|h|^2 \Big]dV_M\\
{} + \|\phi(h,H)\|_{\infty}\bigg\}.
\end{align*}
\end{corollary}
\begin{proof}
Inequality (\ref{r}) gives, for $j=1$
\begin{equation*}
\sum_{l=1}^m \lambda_{l+1}^{(p)} \leq 4\bigg\{
\bigg(1+\frac{m}{4}\bigg)\lambda_1^{(p)}+ \|\phi(h,H)\|_{\infty}
\bigg\}.
\end{equation*}
We finish the proof by using the Asada inequality \cite{Asada}, 
\begin{equation}\label{asada's ineq}
\lambda_1^{(p)}(M) \leq \frac{p}{m(m-1) Vol(M)} \int _M
\Big[(m-p)|H|^2+(p-1)|h|^2\Big] dV_M.
\end{equation}

\end{proof}

For any $j \ge 1$, an immediate consequence of the precedent
corollary is the following upper bound for  $\lambda_{j+m}^{(p)}$ in
terms of the second fundamental form and the mean curvature,
\begin{corollary}
Under the conditions of Theorem \ref{theorem 1}, we have for any $j \ge 1$,
\begin{align}\label{t}
{}\lambda_{j+m}^{(p)} \leq
4\bigg\{d_1(m,j)\frac{p}{Vol(M)}\int_M\Big[(m-p)|H|^2+(p-1)|h|^2
\Big]dV_M \nonumber\\
{} +d_2(m,j) \|\phi(h,H)\|_{\infty}\bigg\},
\end{align}
where
$d_1(m,j)=\frac{1}{m(m-1)}\Big(1+\frac{m}{4}\Big)\Big(1+\frac{4}{m}\Big)^{j-1}$
and\\
$d_2(m,j)=\Big(1+\frac{m}{4}\Big)\Big(1+\frac{4}{m}\Big)^{j-1}-\frac{m}{4}$
\end{corollary}
\begin{proof}
We infer from (\ref{q})
\begin{equation*}
\lambda_{j+m}^{(p)} \leq 4\Big[ \Big(1+\frac{m}{4} \Big)
\lambda_j^{(p)} + \|\phi(h,H)\|_{\infty}\Big],
\end{equation*}
to obtain inequality (\ref{t}), we combine this last inequality with the following inequality obtained by us in \cite{IlMa} (see Corollary 3.6),
\begin{align}\label{s}
{}\lambda_j^{(p)} \leq \frac{1}{m(m-1)} \Big(1+\frac{4}{m}\Big)^{j-1} \frac{p}{Vol(M)} & \int_M \Big[ (m-p)
|H|^2+(p-1)|h|^2 \Big] dV_M
\nonumber\\
{} &+ \Big(\Big(1+\frac{4}{m} \Big)^{j-1}-1 \Big) \|\phi(h,H)\|_{\infty}.
\end{align}

\end{proof}

\begin{remark}
\begin{enumerate}
\item Note that inequality (\ref{t}) is sharper than inequality (\ref{s}) for
$j+m$.
\item We can obtain similar results for closed Submanifolds of compact rank one symmetric spaces but the expressions of $\phi(h,H)$ and $\Phi(h,H)$ in this case are complicated.
\end{enumerate}
\end{remark}
In the particular case where $p=0$ (i.e. for functions), all the arguments used
in the proof of Theorem \ref{theorem 1} work under the Dirichlet boundary condition, when $M$ has
boundary . The reason why all these arguments work in this case is that the product of a function $G$ by a function vanishing on $\partial M$ also vanishes on $\partial M$ (this is neither the case for the absolute boundary condition nor the relative one for $p-$forms when $p\ge1$). Then, we easily obtain

\begin{corollary} \label{cor theorem 1} Let $(M,g)$ be a compact $m$-dimensional Riemannian manifold eventually with boundary and let $X:(M,g)\longrightarrow (\mathbb{R}^{n},{\rm can})$ be an isometric immersion. For any bounded potential $q$ on $M$, the spectrum of $L=\Delta+q$ (with Dirichlet boundary condition if $\partial M \ne \emptyset$) must satisfy, for $j \ge 1$,
 \begin{align*}
{} \displaystyle \sum_{l=1}^m \lambda_{j+l} & \leq 4 \Big[
(1+\frac{m}{4}) \lambda_{j}+\int_{M}\Big(\frac{1}{4}\vert H
\vert^2-q\Big)u_j^2\Big]\\
& \leq (4+m) \lambda_{j}+\left\|\vert H\vert^2-4q\right\|_{\infty},
 \end{align*}
where $u_{j}$ are $L^{2}-$normalized eigenfunctions.
\end{corollary}
This corollary extends, the universal inequality of Levitin and
Parnovski to compact Submanifolds of $\R^n$ and gives for $k=1$ the main result of \cite{ChengChen}.\\

When $M$ is without boundary and $q=0$, we have $\lambda_{1}=0$ and the associated normalized eigenfunction is $u_{1}=\displaystyle{\frac{1}{\sqrt{Vol(M)}}}$; in this case, the Corollary \ref{cor theorem 1} gives the following generalization of Reilly's inequality for the first nonzero eigenvalue of the Laplacian operator on Euclidean closed Submanifolds \cite{Reilly},
\begin{corollary} \label{cor Reilly general} Let $(M,g)$ be a compact $m$-dimensional Riemannian manifold and let $X:(M,g)\longrightarrow (\mathbb{R}^{n},{\rm can})$ be an isometric immersion of mean curvature $H$. Then the spectrum of $\Delta$ must satisfy 
\begin{equation*}
 \displaystyle \sum_{k=1}^m \lambda_{k+1}\leq \frac{1}{Vol(M)}\int_{M} |H|^2.
\end{equation*}
\end{corollary}

As in \cite{Soufi.Harl.Ilias} (Lemma 3.1) or \cite{IlMa}, using the standard
embeddings of rank one compact symmetric spaces, we deduce easily from Corollary \ref{cor theorem 1} the following
\begin{corollary} \label{corRSS} Let $\overline{M}$ be the Sphere $\mathbb{S}^{n}$, the real projective space $\mathbb{R}P^{n}$, the complex projective space $\mathbb{C}P^{n}$ or the quaternionic projective space $\mathbb{Q}P^{n}$ endowed with their respective standard metrics. Let $(M,g)$ be a compact $m$-dimensional Riemannian manifold eventually with boundary and let $X:M\longrightarrow \overline{M}$ be an isometric immersion. For any bounded potential $q$ on $M$, the spectrum of $L=\Delta+q$ (with Dirichlet boundary condition if $\partial M \ne \emptyset$) must satisfy, for $j \ge 1$,
 \begin{align*}
{} \displaystyle \sum_{l=1}^m \lambda_{j+l} & \leq 4 \Big[ \Big(1+\frac{m}{4}\Big) \lambda_{j}+\int_{M}\Big(\frac{1}{4}\big(\vert H \vert^2+d(m)\big)-q\Big)u_j^2\Big]\\
{} & \le (4+m) \lambda_{j}+\left\|\vert H \vert^2+d(m)-4q\right\|_{\infty},
\end{align*}
where $u_{j}$ are $L^{2}-$ normalized eigenfunctions and where
\begin{equation*}
d(m)=
    \begin{cases}
    m^{2},     &\text{if $\overline{M}=\mathbb{S}^{n}$}\\
    2m(m+1),   &\text{if $\overline{M}= \mathbb{R}P^{n}$}\\
    2m(m+2),   &\text{if $\overline{M}= \mathbb{C}P^{n}$}\\
    2m(m+4),    &\text{if $\overline{M}= \mathbb{Q}P^{n}$}.\\

    \end{cases}
\end{equation*}
\end{corollary}

\begin{remark} \label{remRSS}
\begin{enumerate}
\item If we apply this Corollary to a bounded domain of $\mathbb{S}^{n}$ or $\mathbb{C}P^{n}$ and to complex Submanifolds of $\mathbb{C}P^{n}$, then we obtain Theorem 1.1, Theorem 1.2 and Theorem 1.3 of Sun, Cheng and Yang \cite{ChengYang3}. 
\item When $M$ is without boundary and $q=0$, this gives as in corollary \ref{cor Reilly general}, the following generalized Reilly's inequality for compact Submanifolds of compact rank one symmetric spaces (with the exception of the Cayley projective space),
\begin{equation*}
 \displaystyle \sum_{k=1}^m \lambda_{k+1}\leq \frac{1}{Vol(M)}\int_{M}\left(|H|^2+d(m)\right).
\end{equation*}
\end{enumerate}
\end{remark}

\section{Generalization of the Levitin-Parnovski inequality to the Kohn Laplacian on the Heisenberg group}
We first recall that the $2n+1$-dimensional Heisenberg group $\mathbb{H}^{n}$ is
the space $\mathbb{R}^{2n+1}$ equipped with the non-commutative group law
$$ (x,y,t)(x',y',t')=\left(x+x',y+y',t+t'+\frac{1}{2}\right)
(\left\langle x',y\right\rangle_{\mathbb{R}^{n}}-\left\langle x,y'\right\rangle_{\mathbb{R}^{n}}),$$
where $x,x',y,y'\in \mathbb{R}^{n},\; t \;\rm{and}\; t' \in \mathbb{R}$.
The following vector fields 
$$\left\{ T=\frac{\partial}{\partial t},\ X_{i}=\frac{\partial}{\partial x_{i}}+\frac{y_{i}}{2}\frac{\partial}{\partial t},\ Y_{i}=\frac{\partial}{\partial y_{i}}-\frac{x_{i}}{2}\frac{\partial}{\partial t}\ ; \ {i \leq n}\right\}$$
form a basis of the Lie algebra of $\mathbb{H}^{n}$, denoted by $\mathcal{H}^{n}$. We notice that the only non--trivial commutators are $\left[X_{i},Y_{j}\right]= - T \delta_{ij},\; i,j=1, \cdots ,n$.
Let $\Delta_{\mathbb{H}^{n}}$ denote the real Kohn Laplacian (or the sublaplacian associated with the basis 
$\left\{X_{1},\cdots,X_{n},Y_{1},\cdots,Y_{n}\right\}$): 
\begin{align*}
 \Delta_{\mathbb{H}^{n}}&= \sum_{i=1}^{n} X_{i}^{2}+Y_{i}^{2}\\
&= \Delta^{\mathbb{R}^{2n}}_{xy}+\frac{1}{4}(|x|^{2}+|y|^{2})\frac{\partial^{2}}{{\partial t}^{2}}+ \frac{\partial}{\partial t}\sum_{i=1}^{n}\left(y_{i}\frac{\partial}{\partial x_{i}}-x_{i}\frac{\partial}{\partial y_{i}}\right).\end{align*}
We consider the following eigenvalue problem :
\[\begin{cases}
-\Delta_{\mathbb{H}^{n}} {u} = \lambda {u}\,\,\,\, \hbox{in}\ \Omega \\
{u}=0 \,\,\,\ \ \hbox{on} \ \partial\Omega,
\end{cases}\]

\par\noindent
where $\Omega$ is a bounded domain of the Heisenberg group $\mathbb{H}^{n}$, with smooth boundary.
It is known that the Dirichlet problem (5.1) has a discrete spectrum.  
In what follows, we let
$$ 0 < \lambda_{1} \leq \lambda_{2} \leq \cdots \leq \lambda_{k} \cdots  \rightarrow +\infty, $$
denote its eigenvalues and orthonormalize its eigenfunctions $u_{1},\, u_{2},\,\cdots \, \in S^{1,2}_{0}(\Omega)$
so that, $\forall i,j\ge 1$,
$$\left\langle u_{i},u_{j}\right\rangle_{L^{2}}= \int_{\Omega} u_{i} u_{j} dx \, dy \, dt= \delta_{ij}. $$
Here, $S^{1,2}(\Omega)$ denotes the Hilbert space of the functions $u \in L^{2}(\Omega)$ such that $X_{i}(u),\, Y_{i}(u) \in L^{2}(\Omega)$, and $S^{1,2}_{0}$
denotes the closure of $\mathcal{C}^{\infty}_{0}(\Omega)$ with respect to the Sobolev norm
$$ \|u\|^{2}_{S^{1,2}}= \int_{\Omega} \Big(|\nabla_{\mathbb{H}^{n}}u|^{2}+ |u|^{2}\Big) dx\,dy\,dt,$$
with $\nabla_{\mathbb{H}^{n}}u= (X_{1}(u),\cdots, X_{n}(u),Y_{1}(u),\cdots,Y_{n}(u)).$\\

The main result of this paragraph is the following 
\begin{theorem}
 For any $j \ge 1$,
\begin{equation}\label{Lev-Parn}
 \sum_{l=1}^n \lambda_{j+l} \leq (n+2) \lambda_j.
\end{equation}
\end{theorem}
\begin{proof}
Inequality (\ref{Lev-Parn}) follows by applying Lemma \ref{levitin}, with $L=-\Delta_{\mathbb{H}^n}$ and $G=x_l$ or $G=y_l$, $l=1,\ldots,n$.\\
We obtain, as in the proof of Theorem \ref{theorem 1},
\begin{equation}\label{u}
-\frac{1}{2} \langle [[L,x_l],x_l]u_j,u_j \rangle_{L^2} \leq \frac{1}{\lambda_{j+l}-\lambda_j}\|[L,x_l]u_j\|_{L^2}^2
\end{equation}
and \begin{equation}\label{v}
-\frac{1}{2} \langle [[L,y_l],y_l]u_j,u_j \rangle_{L^2} \leq \frac{1}{\lambda_{j+l}-\lambda_j}\|[L,y_l]u_j\|_{L^2}^2.
    \end{equation}
Taking the sum of (\ref{u}) and (\ref{v}) and summing on $l$ from $1$ to $n$ gives
\begin{align}
{} -\frac{1}{2} \sum_{l=1}^n (\lambda_{j+l}-\lambda_j) & \langle [[L,x_l],x_l]u_j,u_j \rangle_{L^2} -\frac{1}{2}\sum_{l=1}^n (\lambda_{j+l}-\lambda_j)\langle[[L,y_l],y_l]u_j,u_j \rangle_{L^2} \nonumber \\
 {} & \leq \sum_{l=1}^n \|[L,x_l]u_j\|_{L^2}^2+\sum_{l=1}^n \|[L,y_l]u_j\|_{L^2}^2. \label{y}
\end{align}
By a straightforward calculation, we obtain 
$$[L,x_l]u_j=-2 X_l(u_j) \quad {\rm and}\quad [L,y_l]u_j=-2 Y_l(u_j).$$
Hence
\begin{equation}\label{z}
\sum_{l=1}^n \|[L,x_l]u_j\|_{L^2}^2+\sum_{l=1}^n \|[L,y_l]u_j\|_{L^2}^2=4 \int_{\Omega}|\nabla_{\mathbb{H}^n}u_j|^2=4\lambda_j.
\end{equation}
Now
\begin{equation}\label{a'}
[[L,x_l],x_l]u_j=-2[X_l,x_l]u_j=-2u_j
\end{equation}
and
\begin{equation}
[[L,y_l],y_l]u_j=-2[Y_l,y_l]u_j=-2u_j.
\end{equation}
Thus
\begin{equation}\label{b'}
 \langle [[L,x_l],x_l]u_j,u_j \rangle_{L^2}=-2 \int_{\Omega}u_j^2=-2
\end{equation}
and
\begin{equation}\label{c'}
 \langle [[L,y_l],y_l]u_j,u_j \rangle_{L^2}=-2 \int_{\Omega}u_j^2=-2.
\end{equation}
Finally, putting identities (\ref{z}), (\ref{b'}) and (\ref{c'}) in (\ref{y}), we obtain inequality (\ref{Lev-Parn}). 
\end{proof}
\subsection*{Acknowledgments}

This work was partially supported by the ANR(Agence Nationale de la Recherche) through FOG project(ANR-07-BLAN-0251-01).  We also wish to
thank the referee for his valuable suggestions which helped us in improving the first presentation of this article.

\end{document}